\documentclass[11pt, a4paper]{amsart}

\usepackage{a4}
\usepackage{amssymb}
\usepackage{amsmath}
\usepackage{amsthm}
\usepackage{amstext}
\usepackage{amscd}
\usepackage{latexsym}
\usepackage{graphics}
\usepackage{color}
\usepackage[all]{xy}

\usepackage[colorlinks, pagebackref]{hyperref}
\hypersetup{
  colorlinks=true,
  citecolor=blue,
  linkcolor=blue,
  urlcolor=blue}

\theoremstyle{plain}
\newtheorem*{mainthm}{Main Theorem}
\newtheorem{thm}{Theorem}[section]
\newtheorem{q}[thm]{Question}

\newtheorem{lem}[thm]{Lemma}
\newtheorem{cor}[thm]{Corollary}

\newtheorem{conj}[thm]{Conjecture}
\theoremstyle{definition}

\newtheorem{conventions}[thm]{Conventions}
\newtheorem{notation}[thm]{Notation}
\newtheorem{rmk}[thm]{Remark}

\newtheorem{defn}[thm]{Definition}

\newcommand{\Hom}{{\rm Hom}}

\newcommand{\Spec}{{\rm Spec \,}}

\newcommand{\Char}{{\rm char \,}}

\newcommand{\sF}{{\mathcal F}}

\newcommand{\sH}{{\mathcal H}}

\newcommand{\sS}{{\mathcal S}}

\newcommand{\sX}{{\mathcal X}}
\newcommand{\sY}{{\mathcal Y}}

\newcommand{\A}{{\mathbb A}}

\newcommand{\G}{{\mathbb G}}

\renewcommand{\P}{{\mathbb P}}

\def\<{\langle}
\def\>{\rangle} 
\def\-{\overline} 
\def\~{\widetilde}
\def\^{\widehat}

\input{xy}
\xyoption{all}

\begin{document}

\author{Chetan Balwe}
\address{School of Mathematics, Tata Institute of Fundamental
Research, Homi Bhabha Road, Mumbai 400005, India}
\email{cbalwe@math.tifr.res.in}

\author{Anand Sawant}
\address{School of Mathematics, Tata Institute of Fundamental
Research, Homi Bhabha Road, Mumbai 400005, India}
\email{anands@math.tifr.res.in}

\thanks{Anand Sawant was supported by the Council of Scientific and Industrial Research, India under the Shyama Prasad Mukherjee Fellowship SPM-07/858(0096)/2011-EMR-I}

\title{$R$-equivalence and $\A^1$-connectedness in anisotropic groups}

\date{}

\begin{abstract}
We show that if $G$ is an anisotropic, semisimple, absolutely almost simple, simply connected group over a field $k$, then two elements of $G$ over any field extension of $k$ are $R$-equivalent if and only if they are $\A^1$-equivalent. As a consequence, we see that $Sing_*(G)$ cannot be $\A^1$-local for such groups.  This implies that the $\A^1$-connected components of a semisimple, absolutely almost simple, simply connected group over a field $k$ form a sheaf of abelian groups.
\end{abstract} 

\maketitle
\section{Introduction}
\label{section-intro}

The notion of \emph{$R$-equivalence} of rational points on a variety, introduced by Manin in 1970's, has been extensively studied in the context of algebraic groups, where it provides a lot of information in the study of rationality properties.  In this note, we explore a connection between the notions of $R$-equivalence in an algebraic group and the sheaf of \emph{$\A^1$-connected components}, in the sense of Morel-Voevodsky. 

Let $G$ be an algebraic group over a field $k$.  If $G$ is an isotropic, semisimple, absolutely almost simple, simply connected group over $k$, classical results can be reinterpreted as saying that we have an isomorphism $G(k)/R \simeq \pi_0^{\A^1}(G)(k)$, where $\pi_0^{\A^1}(G)$ denotes the Nisnevich sheaf of \emph{$\A^1$-connected components} of $G$ (see Theorem \ref{theorem isotropic} below).  In this note, we prove the following result:

\begin{mainthm}
\label{main theorem}
Let $G$ be an anistropic, semisimple, absolutely almost simple, simply connected group over a field $k$ of characteristic $0$.  Let $F$ be a field extension of $k$.  Then the canonical morphism $G(F) \to \pi_0^{\A^1}(G)(F)$ factors through the quotient morphism $G(F) \to G(F)/R$ and induces an isomorphism
\[
G(F)/R \xrightarrow{\simeq} \pi_0^{\A^1}(G)(F).
\]
\noindent Moreover, $Sing_*(G)$ is not $\A^1$-local.  (Here $Sing_*$ denotes the Morel-Voevodsky singular complex construction in $\A^1$-homotopy theory.)
\end{mainthm}

The conditions on $G$ in the statement of the Main Theorem are imposed only because our proof crucially depends on \cite[Th\'eor\`eme 5.8]{Gille}, where they are required.  It seems possible to lift the assumption on the characteristic of the base field (see Remark \ref{remark Gabber}). It may be possible to generalize the Main Theorem to other classes of groups by proving a suitable generalization of \cite[Th\'eor\`eme 5.8]{Gille}. 

This behaviour of anisotropic groups can be contrasted with the behaviour of isotropic groups.  For instance, it has been shown that $Sing_*(G)$ is $\A^1$-local, when $G$ is smooth, split over a perfect field whose semisimple part has fundamental group of order prime to $\Char k$ (see \cite[Proposition 5.11]{Wendt}), and when $G$ is an isotropic reductive group (\cite[Proposition 4.1]{Wendt-Voelkel}).  This allows one to study $\A^1$-connected components of $G$ in terms of \emph{naive} $\A^1$-homotopies.  Indeed, in this case, $\pi_0^{\A^1}(G)(F)$ coincides with $\sS(G)(F)$ and with $G(F)/R$ and can be explicitly described as the quotient of $G(F)$ by its \emph{elementary subgroup} $EG(F)$ (see \cite{Wendt-Voelkel} and Theorem \ref{theorem isotropic} below).

A result of Chernousov-Merkurjev shows that the group of $R$-equivalence classes of a semisimple, absolutely almost simple, simply connected algebraic group over a field is abelian.  This combined with classical results and our Main Theorem shows that for such groups, $\pi_0^{\A^1}(G)$ is a sheaf of abelian groups.

We now briefly outline the contents of this paper.  In Section \ref{section-preliminaries}, we recollect preliminaries on $\A^1$-connectedness and a describe a geometric criterion for two points of an algebraic group to be $\A^1$-equivalent.  In Section \ref{section algebraic groups}, we interpret known results about algebraic groups and $R$-equivalence in the setup of this paper.  These facts are put together in Section \ref{section-proofs} to give a proof of the Main Theorem. 

\section{Preliminaries on \texorpdfstring{$\A^1$}{A1}-connectedness}
\label{section-preliminaries}

Let $k$ be a field and let $Sm/k$ denote the site of smooth schemes of finite type over $k$ along with the Nisnevich topology. We will work with the \emph{$\A^1$-homotopy category} $\sH(k)$ constructed in \cite{Morel-Voevodsky} by inverting all the projection maps of the form $\sX \times \A^1 \to \sX$ in the \emph{simplicial homotopy category} $\sH_s(k)$.  We will follow the notation and terminology used in that paper. In this section, we will briefly recall some ideas from \cite{Balwe-Hogadi-Sawant}. 

For any smooth scheme $U$ over $k$, we say that two morphisms $f,g: U \to \sX$ are \emph{$\A^1$-homotopic} if there exists a morphism $h: U \times \A^1 \to \sX$ such that $h|_{U \times \{0\}} = f$ and $h|_{U \times \{1\}} = g$. We say that $h$ is an \emph{$\A^1$-homotopy} and that it \emph{connects $f$ to $g$}. We say that $f,g: U \to \sX$ are \emph{$\A^1$-chain homotopic} if there exists a finite sequence $f_0=f, \ldots, f_n=g$ such that $f_i$ is $\A^1$-homotopic to $f_{i+1}$, for all $i$. It is easy to see that $\A^1$-chain homotopy is an equivalence relation. 

A simplicial sheaf $\sX$ is said to be \emph{$\A^1$-local} if for any simplicial sheaf $\sY$, the projection map $\sY \times \A^1 \to \sY$ induces a bijection $$\Hom_{\sH_s(k)}(\sY, \sX) \to \Hom_{\sH_s(k)}(\sY \times \A^1, \sX).$$  

There exists an \emph{$\A^1$-localization} endofunctor (\cite[\textsection 2, Theorem 1.66 and p.107]{Morel-Voevodsky}) on the simplicial homotopy category $\sH_s(k)$, denoted by $L_{\A^1}$, such that for every simplicial sheaf $\sX$, the simplicial sheaf $L_{\A^1}(\sX)$ is $\A^1$-local.

We next recall the Morel-Voevodsky singular complex construction $Sing_*$ in $\A^1$-homotopy theory (see \cite[p.87]{Morel-Voevodsky}).  For a simplicial sheaf $\sX$ on $Sm/k$, define $Sing_*(\sX)$ to be the simplicial sheaf given by
\[
Sing_*(\sX)_n = \underline{\Hom}(\Delta_n,\sX_n), 
\]
\noindent where $\Delta_{\bullet}$ denotes the cosimplicial sheaf  
\[
\Delta_n  = \Spec\left(\frac{k[x_0,...,x_n]}{(\sum_ix_i=1)}\right)
\]
\noindent with the natural coface and codegeneracy maps motivated from the ones on topological simplices.

\begin{defn}
\label{definition-S}
Let $\sX$ be a simplicial sheaf on $Sm/k$.  The sheaf of \emph{$\A^1$-chain connected components} of $\sX$ is defined by
\[
\sS(\sX) := \pi_0^s(Sing_*(\sX)),
\] 
\noindent where $\pi_0^s$ of a simplicial sheaf denotes the sheaf of its simplicially connected components.  
\end{defn}

If $X$ is a scheme over $k$, then it is easy to see that $\sS(X)$ is the sheafification in Nisnevich topology of the presheaf on $Sm/k$ that associates with every smooth scheme $U$ over $k$ the set of equivalence classes in $X(U)$ under the relation of $\A^1$-chain homotopy.

\begin{defn}
\label{definition pi0A1}
Let $\sX$ be a simplicial sheaf on $Sm/k$.  The sheaf of \emph{$\A^1$-connected components} of $\sX$ is defined by 
\[
\pi_0^{\A^1}(\sX) := \pi_0^s(L_{\A^1}(\sX)).
\]
\end{defn}

The main obstacle in the study of $\pi_0^{\A^1}$ of a simplicial sheaf is the explicit description of the $\A^1$-localization functor is cumbersome to handle.  The following  result, proved in \cite{Balwe-Hogadi-Sawant}, allows us to use geometric methods in the study of the $\A^1$-connected components sheaf of a smooth scheme over $k$.

\begin{thm}
\label{theorem lim S^n}
Let $\sF$ be a sheaf of sets on $Sm/k$.  Then the sheaf $~\underset{n}{\varinjlim} ~S^n(\sF)$ is $\A^1$-invariant.  Moreover, if $\pi_0^{\A^1}(\sF)$ is $\A^1$-invariant, then the canonical map
\[
\pi_0^{\A^1}(\sF) \to \underset{n}{\varinjlim}~S^n(\sF) 
\]
\noindent is an isomorphism.
\end{thm}


This suggests a method to verify when two sections of a sheaf map to the same element in its $\pi_0^{\A^1}$ (see Lemma \ref{key-lemma} below).  We will use a well-known characterization of Nisnevich sheaves, which we will recall here for the sake of convenience.

For any scheme $U$, an \emph{elementary Nisnevich cover} of $U$ consists of two morphisms $p_1: V_1 \to U$ and $p_2: V_2 \to U$ such that:
\begin{itemize}
 \item[(i)] $p_1$ is an open immersion.
 \item[(ii)] $p_2$ is an \'etale morphism and its restriction to $p_2^{-1}(U \backslash p_1(V_1))$ is an isomorphism onto $U \backslash p_1(V_1)$. 
\end{itemize}
Then a presheaf of sets $\sF$ on $Sm/k$ is a sheaf in Nisnevich topology if and only if the morphism $$\sF(U) \to \sF(V_1) \times_{\sF(V_1 \times_U V_2)} \sF(V_2)$$ is an isomorphism, for all elementary Nisnevich covers $\{ V_1, V_2\}$ of $U$. (See \cite[\textsection 3, Proposition 1.4, p.96]{Morel-Voevodsky} for a proof.)

\begin{lem}
\label{key-lemma}
Let $\sF$ be a sheaf of sets over $Sm/k$ such that the sheaf $\pi_0^{\A^1}(\sF)$ is $\A^1$-invariant. Let $U$ be a smooth scheme over $k$ and let $f,g: U \to \sF$ be two morphisms. Suppose that we are given data of the form 
\[
 \left(\{p_V: V \to \A^1_U, p_W: W \to \A^1_U \}, \{\sigma_0,\sigma_1\}, \{h_V, h_W\}, h \right)
\]
\noindent satisfying the following conditions:
\begin{itemize}
 \item The two morphisms $\{p_V: V \to \A^1_U, p_W: W \to \A^1_U\}$ constitute an elementary Nisnevich cover.
 \item For $i \in \{0,1\}$, $\sigma_i$ is a morphism $U \to V \coprod W$ such that $(p_V \coprod p_W) \circ \sigma_i: U \to U \times \A^1$ is the closed embedding $U \times \{i\} \hookrightarrow U \times \A^1$.
 \item $h_V$ and $h_W$ are morphisms from $V$ and $W$ respectively into $\sF$ such that $(h_V \coprod h_W) \circ \sigma_0 = f$ and $(h_V \coprod h_W) \circ \sigma_1 = g$. 
 \item Let $pr_V: V\times_{\A^1_U} W \to V$ and $pr_W: V \times_{\A^1_U} W \to W$  denote the projection morphisms. Then $h = (h_1, \ldots, h_n)$ is an $\A^1$-chain homotopy connecting the two morphisms $h_V \circ pr_V$ and $h_W \circ pr_W: V \times_{\A^1_U} W  \to \sF$. 
\end{itemize}
Then $f$ and $g$ map to the same element under the map $\sF(U) \to \pi_0^{\A^1}(\sF)(U)$.
\end{lem}

\begin{proof}
The data given above (which, in the terminology of \cite{Balwe-Hogadi-Sawant}, is a special case of an ``$\A^1$-ghost homotopy''), gives rise to a homotopy $H: \A^1_U \to \sS(\sF)$. Indeed, since $\{p_V, p_W\}$ is an elementary Nisnevich cover and since $\sS(\sF)$ is a Nisnevich sheaf, the two compositions $$U \stackrel{h_i}{\to} \sF \to \sS(\sF)$$ can be glued together to give a morphism $\A^1_U \to \sS(\sF)$ which connects the images of $f$ and $g$ in $\sS(\sF)(U)$. Thus $f$ and $g$ map to the same element of $\sS^2(\sF)(U)$.  We have the following commutative diagram:
\[
\xymatrix{
\sS(\sF) \ar[r] \ar[d] & \pi_0^{\A^1}(\sF) \ar[d] \\
\sS^2(\sF) \ar[r] & \varinjlim_n~ \sS^n(\sF)
} 
\]
\noindent Since $\pi_0^{\A^1}(\sF)$ is $\A^1$-invariant, by Theorem \ref{theorem lim S^n} we have $\pi_0^{\A^1}(\sF) \stackrel{\sim}{\to} \underset{n}{\varinjlim}~ \sS^n(\sF)$. Therefore, $f$ and $g$ map to the same element of $\pi_0^{\A^1}(\sF)(U)$. 
\end{proof}

\begin{rmk}
Using the arguments in \cite[Section 4.1]{Balwe-Hogadi-Sawant}, one can see that Lemma \ref{key-lemma} holds even without the hypothesis that $\pi_0^{\A^1}(\sF)$ is $\A^1$-invariant.  However, we make this simplifying assumption since we only need to use it in a situation where $\pi_0^{\A^1}(\sF)$ is known to be $\A^1$-invariant.
\end{rmk}

\section{Algebraic groups and \texorpdfstring{$R$}{R}-equivalence}
\label{section algebraic groups}

\begin{defn}
Let $G$ be an algebraic group over a field $k$.  Two $k$-rational points $x, y$ of $G$ are said to be \emph{$R$-equivalent} if there is a rational map $f: \P^1_k \dashrightarrow G$ defined at $0$ and $1$ such that $f(0)=x$ and $f(1)=y$.
\end{defn}

The relation of $R$-equivalence generates a normal subgroup of $G(k)$ and one denotes the group of $R$-equivalence classes of the set $G(k)$ by $G(k)/R$.  

\begin{notation}
Let $F$ be a field extension of $k$.  We set $$G(F)/R := (G \times_{\Spec k} \Spec F)(F)/R.$$ 
\end{notation}

\begin{defn}
For an algebraic group $G$ over a field $k$ and a field extension $F$ of $k$, let $G(F)^+$ be the subgroup of $G(F)$ generated by the subsets $U(F)$ where $U$ varies over all $F$-subgroups of $G$ which are isomorphic to the additive group $\G_a$.  The group
\[
W(F,G):= G(F)/G(F)^+
\]
\noindent is called the \emph{Whitehead group} of $G$ over $F$.
\end{defn}

We now state an interpretation of the known results in the isotropic case, which will play a crucial role in our proof of the Main Theorem.

\begin{thm}
\label{theorem isotropic}
Let $G$ be an isotropic, semisimple, simply connected, absolutely almost simple group over an infinite field $k$.  Then there is an isomorphism
\[
\pi_0^{\A^1}(G)(k) \simeq G(k)/R. 
\]
\end{thm}
\begin{proof}
Note that the canonical quotient map $G(k) \to G(k)/R$ clearly factors through the map $G(k) \to \sS(G)(k)$.  By \cite[Th\'eor\`eme 7.2]{Gille}, we identify $G(k)/R$ with the Whitehead group $W(k, G)$.  Therefore, any two $R$-equivalent elements of $G(k)$ differ by an element of $G(k)^+$, which gives an $\A^1$-chain homotopy between the two elements.  This shows that $\sS(G)(k) = G(k)/R$.

A result of V\"olkel-Wendt \cite[Corollary 3.4, Proposition 4.1]{Wendt-Voelkel} and Moser (unpublished) says that for an isotropic reductive group $G$, $Sing_*(G)$ is $\A^1$-local.  Therefore, the canonical map $\sS(G) \to \pi_0^{\A^1}(G)$ is an isomorphism.
\end{proof}

We next quote a straightforward consequence of \cite[8.2]{Borel-Tits}.

\begin{thm}[Borel-Tits]
\label{theorem existence of compactification}
Let $G$ be a smooth affine group scheme over a perfect field $k$.  Then the following are equivalent:\\
\noindent $(1)$ $G$ admits no $k$-subgroup isomorphic to $\G_a$ or $\G_m$.\\ 
\noindent $(2)$ $G$ admits a $G$-equivariant compactification $\-G$ such that $G(k) = \-G(k)$.
\end{thm}

We end this section by noting down a few simple observations, which will be useful in the proof of the Main Theorem.

\begin{lem}
\label{lemma rational points}
Let $G$ be an anisotropic group over a perfect field $k$.  Then any rational map $h: \P^1_k \dashrightarrow G$ is defined at all the $k$-rational points of $\P^1_k$.
\end{lem}
\begin{proof}
By Theorem \ref{theorem existence of compactification}, there exists a compactification $\-G$ of $G$ such that $G(k) = \-G(k)$.  Clearly $h$ can be extended to a morphism $\-h: \P^1_{k} \to \-G$ and the lemma follows.  
\end{proof}

\begin{lem}
\label{lemma-compactness}
Let $G$ be an anisotropic group over a perfect field $k$.  Then there are no non-constant morphisms from $\A^1_{k}$ into $G$ and consequently, $$\sS(G)(k) = G(k).$$
\end{lem}
\begin{proof}
Again, obtain a compactification $\-G$ of $G$ such that $G(k) = \-G(k)$ by applying Theorem \ref{theorem existence of compactification}.  Any morphism $h: \A^1_{k} \to G$ can be extended to a morphism $\-h: \P^1_{k} \to \-G$. By Lemma \ref{lemma rational points}, the morphism $\-h$ maps all the $k$-rational points of $\P^1_{k}$ into $G(k)$.  Since $h$ maps every point of $\P^1_{k}$ other than $\infty$ into $G$ anyway, we see that $\-h$ maps $\P^1_{k}$ into $G$ which is an affine scheme. Thus, $\-h$ is the constant map. This shows that $\sS(G)(k) = G(k)$. 
\end{proof}

\section{Proof of the main theorem}
\label{section-proofs}

This section will be devoted to the proof of the Main Theorem stated in the introduction.  We recall that according to \cite[Theorem 4.18]{Choudhury}, for any algebraic group $G$, the sheaf $\pi_0^{\A^1}(G)$ is $\A^1$-invariant. This allows us to use Lemma \ref{key-lemma} in the following proof.

\begin{conventions}
We will use the following conventions in this section:
\begin{itemize}
\item[(1)] For any scheme $X$ over $k$ and any field extension $L/k$, $X_{L}$ will denote the pullback $X \times_{\Spec(k)} \Spec(L)$ over $L$. Similarly, for any morphism $f: X \to Y$ between schemes over $k$, we will denote by $f_{L}: X_{L} \to Y_{L}$ the pullback of $f$ with respect to the projection $Y_{L} \to Y$.  
\item[(2)] For any smooth scheme $U$ over $k$ and any sheaf $\sF$ on $Sm/k$, we will say that $f,g \in \sF(U)$ are \emph{$\A^1$-equivalent} if they map to the same element of $\pi_0^{\A^1}(\sF)(U)$.
\end{itemize}
\end{conventions}

\begin{thm}
\label{theorem-anisotropic}
Let $G$ be an anisotropic, semisimple, absolutely almost simple, simply connected group over a field $k$ of characteristic $0$.  Let $F$ be a field extension of $k$.  Then two elements of $G(F)$ are $R$-equivalent if and only if they are $\A^1$-equivalent.
\end{thm}

\begin{proof}
In view of Theorem \ref{theorem isotropic}, observe that it suffices to prove the theorem in the case $F=k$.

\noindent \emph{Proof of the ``if" part:} By Theorem \ref{theorem existence of compactification}, there exists a compactification $\-G$ of $G$ such that $G(k) = \-G(k)$.  If two elements $p$ and $q$ of $G(k)$ are $\A^1$-equivalent, then $p$ and $q$ map to the same element in $\pi_0^{\A^1}(\-G)(k)$.  Since $\-G$ is proper over $k$, we can apply Theorem \cite[Theorem 2.4.3]{Asok-Morel} to conclude that $p$ and $q$ map to the same element in $\sS(\-G)(k)$.  Therefore, $p$ and $q$ are $\A^1$-chain homotopic $k$-rational points of $\-G$.  Since $\-G(k) \backslash G(k) = \emptyset$, it follows that $p$ and $q$ map to the same element in $G(k)/R$.

\noindent \emph{Proof of the ``only if" part:}
Let $p$ and $q$ be two elements of $G(k)$, which are $R$-equivalent. Thus, there is a rational map $h: \P^1_{k} \dashrightarrow G$ which is defined on $0$ and $1$ such that $h(0) = p$ and $h(1) = q$. Choose a compactification $\-G$ of $G$ such that $G(k) = \-G(k)$.  The rational map $h$ can be uniquely extended to a morphism $\-h: \P^1_{k} \to \-G$.   By Lemma \ref{lemma rational points}, $\-h$ maps all the $k$-rational points of $\P^1_{k}$ into $G$.  Thus, we see that $h$ is undefined only at points of $\A^1_{k}$ having residue fields that are non-trivial finite extensions of $k$.  We define $V := \-h^{-1}(G) \cap \A^1_k$ which is a Zariski open subscheme of $\A^1_k$. Let $\A^1_k \backslash V = \{p_1, \ldots, p_n\}$ and let the residue field at $p_i$ be $L_i$. We define $h_V: V \to G$ by $h_V := \-h|_{V}$. 

We claim that for each $i$, $G_{L_i}$ is an isotropic group. Indeed, the rational map $h_{L_i}:\P^1_{L_i} \dashrightarrow G_{L_i}$ is not defined at an $L_i$-rational point. Hence, by Lemma \ref{lemma rational points}, $G_{L_i}$ cannot be anisotropic.  

Since the group $G_{L_i}$ is isotropic, we may apply \cite[Th\'eor\`eme 5.8]{Gille}, which says that $W(L_i, G) = W(L_i(t), G)$. Thus any element of $G_{L_i}(L_i(t))$ can be connected by an $\A^1$-chain homotopy to an element in the image of the natural map $G_{L_i}(L_i) \to G_{L_i}(L_i(t))$. Applying this to the map $(h_V)_{L_i}: V_{L_i} \to G_{L_i}$, we see that there exists some open subscheme $V_i^{\prime}$ of $V_{L_i}$ such that the map $(h_{V})_{L_i}|_{V_i^{\prime}}$ can be connected by an $\A^1$-chain homotopy to a constant map taking $V_i^{\prime}$ to some $L_i$-rational point $q_i^{\prime} \in G_{L_i}(L_i)$. 

Choose a preimage $p_i'$ of $p_i$ under the projection map $\A^1_{L_i} \to \A^1_k$ for each $i$ and denote by $V_i$ the open subscheme of $\A^1_{L_i}$ given by $V_{i}' \cup \{ p_i'\}$. Let $q_i$ be the image of $q_i^{\prime}$ under the projection $G_{L_i} \to G$. We define $h_i: V_i \to G$ to be the constant map taking $V_i$ to the point $q_i$. Let $W:= \coprod_i V_i$ and let $h_W: W \to G$ be the map $\coprod_i h_i$. 

We define $p_V: V \to \A^1_k$ to be the inclusion. For each $i$, we define $p_i: V_i \to \A^1_k$ to be the composition $V_i \hookrightarrow \A^1_{L_i} \to \A^1_k$. Let $p_W: W \to \A^1_k$ be the map $\coprod_i p_i$.  Since $p_W^{-1}(\A^1_k \backslash V) = \{ p_1', \ldots, p_n'\}$, it is easy to see that $\{p_V, p_W\}$ is an elementary Nisnevich cover of $\A^1_k$. In order to apply Lemma \ref{key-lemma}, we need to show that the morphisms $h_V \circ pr_V$ and $h_W \circ pr_W$ from  $V \times_{\A^1_k} W $ to $G$ are $\A^1$-chain homotopic. 

For every $1 \leq i \leq n$, we have $V \times_{\A^1_k} V_i = V_i^{\prime}$. Thus $V \times_{\A^1_k} W = \coprod_i V_i^{\prime}$. The morphism $pr_V|_{V_i^{\prime}}$ is equal to the composition $V_i^{\prime} \hookrightarrow V_{L_i} \rightarrow V$. Also, the morphism $pr_W|_{V_{i}^{\prime}}$ is equal to the composition of inclusions $V_i^{\prime} \subset V_i \subset W$. 

For each $i$, we have the commutative diagrams
\[
\xymatrix{
V_i^{\prime} \ar@{^{(}->}[r] & V_{L_i} \ar[r]^{(h_V)_{L_i}} \ar[d] & G_{L_i} \ar[d] \\
                             & V       \ar[r]_{h_V}                & G 
}
\]
and
\[
\xymatrix{
V_i^{\prime} \ar@{^{(}->}[r] & V_i \ar[r]^{c_{q^{\prime}_i}} \ar[rd]_{h_W|_{V_i}}& G_{L_i} \ar[d] \\
                             &                                          & G     
}
\]
where $c_{q_i^{\prime}}$ is the constant map taking the scheme $V_{L_i}$ to $q_{i}^{\prime}$. By assumption, there exists an $\A^1$-chain homotopy connecting the maps $(h_V)_{L_i}|_{V_i^{\prime}}$ to the map $c_{q_i^{\prime}}|_{V_i^{\prime}}$. On composing with the projection map $G_{L_i} \to G$, this gives an $\A^1$-chain homotopy connecting the morphism $h_V \circ pr_V|_{V_{i}^{\prime}}$ to the morphism $h_W \circ pr_W|_{V_i^{\prime}}$. Thus, there exists an $\A^1$-chain homotopy connecting the morphisms $h_V \circ pr_V$ to the morphism $h_W \circ pr_W$. 

Thus, we may now apply Lemma \ref{key-lemma} to conclude that $p$ and $q$ map to the same element in $\pi_0^{\A^1}(G)(k)$.  This completes the proof of Theorem \ref{theorem-anisotropic}.
\end{proof}

\begin{rmk}
\label{remark Gabber}
An unpublished result of Gabber generalizes Theorem \ref{theorem existence of compactification} to fields that are not perfect and to groups that are not necessarily smooth.  This can be used to generalize Theorem \ref{theorem-anisotropic} to fields that are not perfect by closely following the proof of Theorem \ref{theorem-anisotropic}.  The only adjustment needed is in the proof of the ``if" part, where one replaces the use of \cite[Theorem 2.4.3]{Asok-Morel} with the use of \cite[Theorem 2 in the Introduction]{Balwe-Hogadi-Sawant}.
\end{rmk}

\begin{cor}
Let $G$ be as in Theorem \ref{theorem-anisotropic}.  Then $Sing_*(G)$ cannot be $\A^1$-local. 
\end{cor}
\begin{proof}
We simply note that there does exist a pair of distinct $R$-equivalent elements in $G(k)$. Indeed, this is an immediate consequence of the fact that $G$ is unirational over $k$ (see \cite[Theorem 18.2]{Borel}). Thus, the map $\sS(G)(k) \to \pi_0^{\A^1}(G)(k)$ is not a bijection. This shows that $Sing_*(G)$ cannot be $\A^1$-local.
\end{proof}

This completes the proof of the Main Theorem.  

\begin{rmk}
\label{remark R-equivalence abelian}
A long-standing open question in the study of $R$-equivalence asks if the group of $R$-equivalence classes $G(k)/R$ of a reductive algebraic group is always abelian.  This has been proved by Chernousov and Merkurjev (see \cite[Th\'eor\`eme 7.7]{Gille} and \cite[1.2]{Chernousov-Merkurjev}) in the case when $G$ is a semisimple, simply connected, absolutely almost simple and of classical type over $k$.
\end{rmk}

Thus, it is natural to conjecture the following.

\begin{conj} 
\label{conjecture R-equivalence}
Let $G$ be a reductive algebraic group over a field $k$.  Then $\pi_0^{\A^1}(G)(F) = G(F)/R$, for all field extensions $F$ of $k$.
\end{conj}

We end with a question posed by Anastasia Stavrova, which is open:

\begin{q} \label{question abelian}
Let $G$ be a reductive algebraic group over a field $k$.  Is $\pi_0^{\A^1}(G)$ a sheaf of abelian groups?
\end{q}

\begin{rmk}
We briefly explain how giving an affirmative answer to Question \ref{question abelian} is equivalent to giving an affirmative answer to the question of abelian-ness of the group of $R$-equivalence classes of a reductive algebraic group $G$ over a field $k$, if Conjecture \ref{conjecture R-equivalence} holds.  One implication is obvious.  For the other, observe that if $G(F)/R$ is abelian for any field extension $F/k$, to answer Question \ref{question abelian} affirmatively,  it suffices to prove that $\pi_0^{\A^1}(G)(\Spec A)$ is an abelian group for regular henselian rings $A$ containing $k$.  This follows from \cite[Corollary 4.17]{Choudhury}, which implies that $\pi_0^{\A^1}(G)(\Spec A)$ injects into $\pi_0^{\A^1}(G)(\Spec Q(A))$, where $Q(A)$ denotes the quotient field of $A$.  This proves the other implication.  This gives an affirmative answer to Question \ref{question abelian} in the case when $G$ is a semisimple, simply connected, absolutely almost simple and of classical type over a field $k$ (see Remark \ref{remark R-equivalence abelian} above).
\end{rmk}


\section*{Acknowledgements}
This note was inspired by Question \ref{question abelian}.  We thank Anastasia Stavrova for posing the question in the AIM Workshop on ``Projective modules and $\A^1$-homotopy theory" at Palo Alto; we also thank the AIM for hospitality.  It is a pleasure to thank M. V. Nori for discussions.

\end{document}